\documentclass[12pt,leqno,amsfonts,amscd]{amsart}
\usepackage{epsfig}
\usepackage{amsmath}
\usepackage{amsfonts}
\usepackage{amssymb}
\usepackage{color}
\usepackage{hyperref}

\newtheorem{theorem}{Theorem}[section]
\newtheorem{lemma}[theorem]{Lemma}
\newtheorem{prop}[theorem]{Proposition}
\newtheorem{cor}[theorem]{Corollary}

\theoremstyle{definition}
\newtheorem{definition}[theorem]{Definition}
\newtheorem{example}[theorem]{Example}
\newtheorem{question}[theorem]{Question}

\theoremstyle{remark}

\newtheorem{remark}[theorem]{Remark}

\numberwithin{equation}{section}

\newcommand{\calf}{{\mathcal F}}

\newcommand{\NN}{{\mathbb N}}

\newcommand{\RR}{{\mathbb R}}
\newcommand{\CC}{{\mathbb C}}

\newcommand{\out}[1]{\ }

 \DeclareMathOperator{\psh}{PSH}
 
 \DeclareMathOperator{\FPSH}{\calf-PSH}
 
 \DeclareMathOperator{\FMPSH}{\calf-MPSH}
\DeclareMathOperator{\QB}{QB}
\DeclareMathOperator{\NP}{NP}

\let\PSH=\psh

\let\cal=\mathcal
\renewcommand{\phi}{\varphi}

\begin{document}
\title[Maximal $\calf$-Plurisubharmonic functions]{Maximal Plurifinely Plurisubharmonic functions}

\author{Mohamed El Kadiri}
\address{Universit\'e Mohammed V-Agdal
\\Facult\'e des Sciences
\\D\'epartement de Math\'ematiques
\\B.P. 1014, Rabat
\\Morocco}
\email{elkadiri@fsr.ac.ma}

\author{Iris M. Smit}
\address{KdV Institute for Mathematics
\\University of Amsterdam
\\Science Park 904
\\P.O. box 94248, 1090 GE Amsterdam
\\The Netherlands}
\email{i.m.smit@uva.nl}



\begin{abstract}
The main purpose of this paper is to introduce and study the
notion of $\calf$-maximal ${\cal F}$-plurisubharmonic functions,
which extends the notion of maximal plurisubharmonic functions on
a Euclidean domain to an $\cal F$-domain of $\CC^n$ in a natural
way. Our main result is that a finite $\cal F$-plurisubharmonic
function $u$ on a plurifine domain $U$ satisfies $(dd^cu)^n=0$ if
and only if $u$ is $\calf$-locally $\calf$-maximal outside some
pluripolar set. In particular, a finite $\calf$-maximal
plurisubharmonic function $u$ satisfies $(dd^c u)^n=0$.

\end{abstract}
\maketitle

\section{Introduction}

The plurifine topology $\calf$ on a Euclidean open set
$\Omega\subset\CC^n$ is the smallest topology that makes all
plurisubharmonic functions on $\Omega$ continuous. This
construction is completely analogous to the better known fine
topology in classical potential theory of H.~Cartan. Good
references for the latter are \cite{AG, D}. The topology $\calf$
was introduced in \cite{F6}, and studied e.g.\ by Bedford and
Taylor in \cite{BT}, and by El Marzguioui and Wiegerinck in
\cite{E-W1, EW1}, where they proved in particular that this
topology is locally connected. Notions related to the topology
$\calf$ are provided with the prefix $\calf$, e.g.\ an
$\calf$-domain is an $\calf$-open set that is connected in
$\calf$.

Just as one introduces finely subharmonic functions on fine
domains in $\RR^n$, cf.~Fuglede's book \cite{F1}, one can
introduce plurifinely plurisubharmonic functions on
$\calf$-domains in $\CC^n$. These functions are called ${\cal
F}$-plurisubharmonic. In case $n=1$, we merely recover the finely
subharmonic functions on fine domains in $\RR^2$.

The definition of ${\cal F}$-plurisubharmonic functions on an
${\cal F}$-open set of $\CC^n$ was first given in \cite{EK} and
\cite{EW2}, where some properties of these functions were studied.
The $\calf$-continuity of the ${\cal F}$-plurisubharmonic
functions was established in \cite{EW2}.

In \cite{EFW}, the most important properties of the ${\cal
F}$-plurisubharmonic functions were obtained. This paper included
a convergence theorem, and the characterization of ${\cal
F}$-plurisubharmonic functions as $\calf$-locally bounded finely
subharmonic functions with the property that they remain finely
subharmonic under composition with $\CC$-isomorphisms of $\CC^n$.
Hence the most important properties of plurisubharmonic functions
on Euclidean opens of $\CC^n$ were extended to ${\cal
F}$-plurisubharmonic functions on ${\cal F}$-open sets of $\CC^n$.

In \cite{EW}, the authors obtained a local approximation of ${\cal
F}$-plurisubharmonic functions by plurisubharmonic functions,
outside a pluripolar set. They also defined the Monge-Amp\`{e}re
measure for finite ${\cal F}$-plurisubharmonic functions on an
${\cal F}$-domain $U$. This construction was based on the fact
(established in \cite{EW2}) that such a function can be
$\calf$-locally represented as a difference between two bounded
plurisubharmonic functions defined on some Euclidean open set, and
the quasi-Lindel\"{o}f property of the plurifine topology. The
local approximation property allowed them to prove that this
Monge-Amp\`{e}re measure is a positive Borel measure on $U$ which
is $\calf$-locally finite and doesn't charge the pluripolar sets.
It is $\sigma$-finite by the quasi-Lindel\"{o}f property of the
plurifine topology.

\medskip

In the theory of plurisubharmonic functions on a Euclidean domain,
the so-called maximal functions are the analog of the harmonic
functions in classical potential theory. They play an important
role in the resolution of the Dirichlet problem for the
Monge-Amp\`ere operator. In this paper we introduce and study the
notion of $\calf$-maximal ${\cal F}$-plurisubharmonic functions,
extending the notion of maximal plurisubharmonic functions on a
Euclidean domain to an $\cal F$-domain of $\CC^n$ in a natural
way. In Section 2 we define $\calf$-maximal
$\calf$-plurisubharmonic functions, and look at some basic
properties of these functions. In the next section we look at the
possibility of adapting plurisubharmonic functions to become
$\calf$-maximal at some $\calf$-open subset. Finally, Section 4
connects the $\calf$-maximality of functions to the
Monge-Amp\`{e}re operator. In particular, we prove that a finite
$\cal F$-plurisubharmonic function $u$ on a plurifine domain $U$
satisfies $(dd^cu)^n=0$ if and only if $u$ is $\calf$-locally
$\calf$-maximal outside some pluripolar set.

\section{Maximal ${\cal F}$-plurisubharmonic functions}
\label{sec1}

In analogy with maximal plurisubharmonic functions, which play a
role in pluripotential theory comparable to that of harmonic
functions in classical potential theory, we will introduce
$\calf$-maximal ${\cal F}$-plurisubharmonic functions. These
relate similarly to finely harmonic functions and constitute the
plurifine analog of maximal plurisubharmonic functions on
Euclidean open sets.

For this article, let $n$ be an integer $\geq 1$. If $A \subseteq
\CC^n$, we denote the closure of $A$ in the Euclidean, fine and
plurifine topologies by $\overline{A}$, $\tilde{A}$ and
$\overline{A}^{\calf}$ respectively. For a function $f$ on $A$
with values in ${\overline \RR}$, we denote by $\limsup_{x\in
A,x\to y} f(x)$, $f$-$\limsup_{x \in A, x \to y} f(x)$ and $\cal
F$-$\limsup_{x\in A,x \to y} f(x)$ the $\limsup$ with respect to
the Euclidean topology of $\CC^n$, the fine topology of $\CC^n
\simeq \RR^{2n}$, and the plurifine topology of $\CC^n$
repectively, and likewise for other limits. The set of
$\calf$-plurisubharmonic functions on an $\calf$-open set $U$ will
be denoted by $\FPSH(U)$.

\begin{definition}\label{def2.1}
Let $U\subset \CC^n$ be an $\calf$-open set and let $u\in
\FPSH(U)$. We say that $u$ is ${\cal F}$-maximal if for every
bounded $\calf$-open set $G$ of $\CC^n$ such that $\overline
G\subset U$, and for every function $v\in \FPSH(G)$ that is
bounded from above on $G$ and extends ${\cal F}$-upper
semicontinuously to ${\overline G}^{\cal F}$, the following holds:
\[
v\le u \ \text{on} \ {\partial }_{\cal F}G \Longrightarrow v\le u \ \text{on} \ G.
\]

We denote by  $\FMPSH(U)$ the set of $\calf$-maximal ${\cal F}$-plurisubharmonic
functions on $U$.
\end{definition}

\begin{prop}\label{prop2.2echt}
Let $U \subseteq \CC^n$ be an $\calf$-domain, and let $(u_j)$ be a
sequence of functions in $\FMPSH(U)$, decreasing to $u$. Then
either $u \equiv - \infty$, or $u \in \FMPSH(U)$.
\end{prop}

\begin{proof}
In view of \cite[p. 84]{F1}, it is easy to see that $u \equiv -
\infty$ or $u \in \FPSH(U)$. In this second case: let $G$ be a
bounded $\calf$-open set such that $\overline{G} \subset U$, and
let $v \in \FPSH(G)$ be bounded from above on $G$, and extend
$\calf$-upper semicontinuously to $\partial_{\calf}G$. Suppose
that $v \leq u$ on $\partial_{\calf}G$. Then for all $j$ we have
$v \leq u \leq u_j$ on $\partial_{\calf}G$, hence $v \leq u_j$ on
$G$ by $\calf$-maximality of $u_j$. This implies that $v \leq u$
on $G$, so $u$ is $\calf$-maximal.
\end{proof}

The next proposition generalizes a result concerning finely
subharmonic functions due to Fuglede, cf.~\cite[Lemma 10.1]{F1},
to the plurifine situation.

\begin{prop}\label{prop2.2}
Let  $G$ and $U$ be ${\cal F}$-open sets in $\CC^n$ such that $G
\subseteq U$. Suppose that $u\in \FPSH(U)$, $v\in\FPSH(G)$, and
$\calf$-$\limsup\limits_{z\in G, z\to \zeta} v(z)\le u(\zeta)$ for
all $\zeta\in {\partial_{\cal F} G} \cap U$. Then the function
\begin{equation*}
w =
\begin{cases}
\max(u,v)&\text{on $G$,}\\
u& \text{on $\ U\setminus G$,}
\end{cases}
\end{equation*}
is ${\cal F}$-plurisubharmonic on $U$.
\end{prop}

\begin{proof}
Let $f:{\CC}^n\longrightarrow {\CC}^n$
be a ${\CC}$-affine bijection. Then
\begin{equation*}
w\circ f =
\begin{cases}
{\max} (u\circ f,v\circ f)&\text{on $ f^{-1}(G)$}, \\
u\circ f                  & \text{on $ f^{-1}(U)\setminus f^{-1}(G)$}.
\end{cases}
\end{equation*}
By \cite[Theorem 3.1]{EFW}, the functions $u\circ f$ and $v\circ
f$ are finely subharmonic on, respectively, the fine open sets
$f^{-1}(U)$ and $f^{-1}(G)$ in ${\CC}^n\cong {\RR}^{2n}$. In view
of the assumptions we have
$$f\textrm{-}\limsup _{z\to \zeta} (v\circ f)(z)\le \calf \textrm{-} \limsup_{z\to \zeta} (v\circ f)(z)\le (u\circ f)(\zeta)$$
for all $\zeta\in {\partial}_f(f^{-1}(G))\subset
\partial_{\calf}(f^{-1}(G))=f^{-1}(\partial_{\calf}G)$, where the
fine limit is lower than or equal to the plurifine limit since the
fine topology is finer. Now Lemma 10.1 in \cite{F1} states that
$w\circ f$ is finely subharmonic on $f^{-1}(U)$. As the function
$w$ is $\calf$-locally bounded, Theorem 3.1 in \cite{EFW} now
shows that $w$ is  ${\cal F}$-plurisubharmonic on $U$.
\end{proof}

\begin{prop}\label{prop2.3}
Let $\Omega \subseteq \CC^n$ be a Euclidean domain. A function
$u\in \PSH(\Omega)$ is ${\cal F}$-maximal if and only if $u$ is
maximal as  a plurisubharmonic function on $\Omega$.
\end{prop}

\begin{proof}
Suppose that  $u\in \PSH(\Omega)$ is ${\cal F}$-maximal. Let $G$
be a bounded open set in ${\CC}^n$ such that $\overline G\subset
\Omega$, and let $v\in \PSH(G)$ be upper semicontinuous on
${\overline G}$ with $v\le u$ on ${\partial G}$. Observe that
$u\in\FPSH(\Omega)$ and $v\in\FPSH(G)$, and that $v$ is bounded
from above on $\overline{G}$. Then, because ${\partial_{\calf}
G}\subset {\partial G}$, we have $v\le u$ on ${\partial_{\calf}
G}$, hence $v\le u$ on $G$ by $\calf$-maximality. In other words,
$u$ is maximal.

For the other implication, suppose that $u$ is maximal. Let $G$ be
a bounded $\calf$-open set with the property that $\overline
G\subset \Omega$, and let $v\in \FPSH(G)$ be bounded from above on
$G$, ${\cal F}$-upper semicontinuous on $\overline{G}^{\calf}$,
and satisfy $v\le u$ on ${\partial_\calf G}$. By Proposition
\ref{prop2.2}, the function
\begin{equation*}
w =
\begin{cases}
\max (u,v)&\text{on $G$,}\\
u& \text{on  $\Omega\setminus G$}
\end{cases}
\end{equation*}
lies in $\FPSH(\Omega)$, and by \cite[Proposition 2.14]{EFW}
therefore also in $\PSH(\Omega)$. Obviously, $w \leq u$ on $\Omega
\setminus G$.
Since $\overline{G}$ is compact, $\CC^n \setminus
\Omega$ is closed, and their intersection is empty, these sets
have a positive distance $d=d(\overline{G},\CC^n \setminus
\Omega)$. (In case $\Omega=\CC^n$, just use $d=1$ instead). Define
$W=\{x \in \Omega : d(x, \overline{G}) < \frac{d}{2} \}$. Then $W$
is Euclidean open and $\overline{G} \subseteq W \subseteq
\overline{W} \subseteq \{x \in \Omega : d(x, \overline{G}) \leq
\frac{d}{2} \} \subseteq \Omega$. Now we have $w \leq u$ on
$\partial W$ and hence by maximality of $u$ on $\Omega$, we can
conclude that $w \leq u$ on $W$. This implies that $v \leq u$ on
$G$, so $u$ is $\calf$-maximal.
\end{proof}

\begin{prop}\label{prop2.5}
Let $U \subseteq \CC^n$ be an $\calf$-open set, and $u \in
\FPSH(U)$ a bounded function that is $\calf$-maximal on $U
\setminus F$ for some pluripolar set $F$. Then $u$ is
$\calf$-maximal on $U$.
\end{prop}

\begin{proof}
By subtracting a constant from $u$, we may assume without loss of
generality that $m<u <0$ for some $m \in \mathbb{R}_{-}$. Let $G$
be a bounded $\calf$-open set such that $G \subseteq \overline{G}
\subseteq U$, and let $v$ be an $\calf$-plurisubharmonic function
on $G$ that is bounded from above on $G$, extends $\calf$-upper
semicontinuously to $\overline{G}^{\calf}$, and satisfies $v \leq
u$ on $\partial_{\calf} G$. As $u<0$, we see that $v < 0$ on
$\partial_{\calf}G$. By Proposition \ref{prop2.2}, the function
$w$ defined below is $\calf$-plurisubharmonic on $\CC^n$.
\begin{equation*}
w =
\begin{cases}
\max(v,0)&\text{on $G$,}\\
0 & \text{on $\ \CC^n \setminus G$,}
\end{cases}
\end{equation*}
By \cite[Proposition 2.14]{EFW}, we see that $w \in \PSH(\CC^n)$.
As $G$ is bounded, the maximum principle gives that $w=0$, and
therefore $v \leq 0$ on $G$.

Since $F$ is pluripolar, we can find $g \in \PSH(\CC^n)$ such that
$F \subset \{g= - \infty\}$. Since $\overline{G}$ is bounded, we
can even pick a $g$ that satisfies $g<0$ on $\overline{G}$. For $k
\in \mathbb{N}$, set $G_k=\{ z \in G : g(z) > km \}\subseteq G$.
Since $g$ is upper semicontinuous and $\calf$-continuous, $G_k$ is
$\calf$-open and $\overline{G_k} \subseteq \{z \in \overline{G}:
g(z) \geq km\} \subseteq U \setminus F$. Note also that as $m<0$,
we have $G_k \subseteq G_{k+1}$. As $G_k=G \cap \{z \in \CC^n :
g(z) >km\}$, we find that $\partial_{\calf} G_k \subseteq
\partial_{\calf} G \cup \partial_{\calf} \{z \in \CC^n : g(z) > km
\}$, and $\partial_{\calf}G_k \subseteq \overline{G}^{\calf}$.

The function $v + \frac{g}{k}$ is $\calf$-plurisubharmonic on
$G_k$, bounded from above by 0 on $G_k$, and $\calf$-upper
semicontinuous on $\overline{G_k}^{\calf}$. On $\partial_{\calf}
G$ we have $v + \frac{g}{k}\leq v+0\leq u$. As
$\partial_{\calf}\{z \in \CC^n : g(z)>km\} \subseteq \{z \in \CC^n
: g(z)=km\}$, we find that on $\partial_{\calf}\{z \in \CC^n :
g(z)>km\} \cap \overline{G}^{\calf}$ we have $v +
\frac{g}{k}=v+\frac{km}{k} \leq 0 +m \leq u$. Hence $v +
\frac{g}{k} \leq u$ on $\partial_{\calf} G_k$. Since
$\overline{G_k} \subseteq U \setminus F$, we can apply the
$\calf$-maximality of $u$ on this set to find that $v +
\frac{g}{k} \leq u$ on $G_k$.

As the sets $G_k$ are increasing, the above inequality also
implies that $v + \frac{g}{k} \leq u$ on $G_{k_0}$ for any $k \geq
k_0$. Letting $k$ tend to infinity, we find that $v \leq u $ on
$G_{k_0}$. Letting $k_0$ tend to infinity shows that $v \leq u$ on
$\bigcup_{k=1}^{\infty} G_k = G \setminus \{g= -\infty\}$. Since
the pluripolar set $\{g=-\infty\}$ has an empty $\calf$-interior
and $u$ and $v$ are $\calf$-continuous, this implies that $v\leq
u$ on $G$. This proves the $\calf$-maximality of $u$ on $U$.
\end{proof}

We can also define an $\calf$-local concept of $\calf$-maximality:

\begin{definition}\label{def3.19}
An ${\cal F}$-plurisubharmonic function on an $\calf$-open set $U
\subseteq \CC^n$ is said to be ${\cal F}$-locally $\calf$-maximal
if each point of $U$ has an ${\cal F}$-open neighborhood  $V
\subseteq U$ such that the restriction of  $f$ to $V$ is ${\cal
F}$-maximal.
\end{definition}

This definition can be used to formulate a corollary to
Proposition \ref{prop2.5}:

\begin{cor}
Let $U \subseteq \CC^n$ be an $\calf$-open set, and $u \in
\FPSH(U)$ a finite function that is $\calf$-maximal on $U
\setminus F$ for some pluripolar set $F$. Then there exists an
increasing sequence $(V_j)$ of $\calf$-open sets in $U$ such that
$\bigcup_{j \in \NN} V_j =U$ and $u$ is $\calf$-maximal on each
$V_j$. In particular, $u$ is $\calf$-locally $\calf$-maximal.
\end{cor}

\begin{proof}
For each $j \in \NN$, set $V_j=\{z\in U: -j<u(z)<j \}$. By
$\calf$-continuity of $u$, the sets $V_j$ are $\calf$-open. Since
$u$ is finite, $\bigcup_{j \in \NN} V_j = U$. Since $u$ is
$\calf$-maximal on $U \setminus F$, it will be bounded and
$\calf$-maximal on $V_j \setminus F$. By Proposition
\ref{prop2.5}, $u$ has to be $\calf$-maximal on all of $V_j$. As
the $V_j$ cover $U$, $u$ will be $\calf$-locally $\calf$-maximal
on $U$.
\end{proof}

Let $f$ be a non-positive ${\cal F}$-plurisubharmonic  function in
an ${\cal F}$-open subset $U$ of $\CC^n$ and $A\subset U$. We put
$$f_A(z)=\sup\{u(z): \ u\in \FPSH_-(U) \textrm{ and } u\le f \ \ {\rm in} \ A\},$$
where $\FPSH_{-}(U)$ denotes the  cone of non-positive ${\cal
F}$-plurisubharmonic functions  on $U$. In \cite[Theorem 3.9]{EFW}
it was shown that the \emph{$\calf$-upper semicontinuous
regularization} $f_A^*$ of $f_A$ is an ${\cal F}$-plurisubharmonic
function in $U$. It will be called the \emph{maximalized function
of $f$ in $A$}.

If $U_1\subset U_2\subset U$ are two ${\cal F}$-open subsets of
$U$ and $A\subset U_1$, we denote by ${^{U_1}f}_A^*$ and
${^{U_2}f}_A^*$ the maximalized functions on $A$ of the
restrictions of $f$ to $U_1$ and $U_2$, respectively. Then we have
${^{U_2}f}_A^* \le {^{U_1}f}_A^*$ on $U_1$.

\medskip

This notion is somewhat similar to that of reduced and swept-out
functions in the fine setting: Let $f \geq 0$ be a function on an
$f$-open set $U \subseteq \RR^n$, and $A \subseteq U$ a subset,
then
$$R^A_f(x) = \inf \{u(x) : u \textrm{ is finely hyperharmonic on }U \textrm{, } u \geq 0\textrm{, and } u\geq f \textrm{ on } A\}$$
and $\widehat{R}^A_f$ is its $f$-lower semicontinuous regularization. See for example \cite{F1}.

\begin{example}[An ${\cal F}$-maximal plurisubharmonic function]\label{example2.4}

Let $U$ be a bounded $\calf$-open set in $\CC^n$, let $u\in
\PSH_-(U)$, and let $V \subset U$ be ${\cal F}$-open. Then $u_{U
\setminus V}^*$ as defined above is an element of $\FMPSH(V)$, as
we will prove in Proposition \ref{prop3.6}.
\end{example}

In the one-dimensional case, $\calf$-maximality and $\calf$-local
$\calf$-maximality are identical for finite functions:

\begin{prop}\label{prop2.9}
Let $U \subseteq \CC$ be an $f$-open set, and let $f$ be a finite
finely subharmonic function on $U$. Then $f$ is $\calf$-maximal if
and only if $f$ is $\calf$-locally $\calf$-maximal, if and only if
$f$ is finely harmonic.
\end{prop}

\begin{proof}
Note that the fine and plurifine topologies are equal in this
setting. Suppose first that $f$ is finely harmonic on $U$. Let $G$
be a bounded $f$-open set such that $\overline {G} \subseteq U$.
Let $v$ be a finely subharmonic function on $G$ that is bounded
from above on $G$, extends $f$-upper semicontinuously to
$\tilde{G}$, and satisfies $v \leq f$ on $\partial_{f} G$. Then
$v-f$ is finely subharmonic on $G$, and by Proposition
\ref{prop2.2}, the function $w$ defined below is finely
subharmonic on $\CC$:
\begin{equation*}
w =
\begin{cases}
\max(v-f,0)&\text{on $G$,}\\
0& \text{on $\ \CC \setminus G$,}
\end{cases}
\end{equation*}
By \cite[Proposition 2.14]{EFW}, we see that $w$ is subharmonic on
$\CC$. As $G$ is bounded, we can use the maximum principle to see
that $w = 0$, and therefore $v \leq f$ on $G$. This proves the
$\calf$-maximality of $f$.

Now assume that $f$ is $\calf$-maximal on $U$. Note that $f$ is
$f$-continuous. We can find a base for the fine topology on $U$,
consisting of bounded, regular sets $V$ such that $V \subseteq
\overline{V} \subseteq U$, and $f$ is bounded on $\overline{V}$.

Let $V$ be a set as described above, and define $v(z)= \int f d
\varepsilon_z^{\mathbf{C}V}$ for $z \in V$. Here,
$d\varepsilon_z^{\mathbf{C}V}$ is the swept-out of the Dirac
measure $\varepsilon_z$ on $\mathbf{C}V$. See \cite{F1} for more
information. Since $f$ is bounded on $\overline{V}$, we can apply
\cite[Theorem 14.6]{F1} to see that $v$ is finely harmonic on $V$.
Since $V$ is relatively compact in $U$, we can find a bounded
domain $D$ such that $\overline{G} \subseteq D $. There exists a
strictly positive potential $p$ on $D$. Since $p$ is lower
semicontinuous, $p$ assumes a positive minimum on $\overline{V}$.
Multiplying $p$ by an appropriate constant will give $|f| \leq p$
on $\overline{V}$, since $f$ is bounded on $\overline{V}$. We can
now apply \cite[Theorem on page 177]{F1} to see that $f$-$\lim_{y
\rightarrow z} v(y) = f(z)$ for all $z \in \partial_{f}V$, since
$f$ is finely continuous and $V$ is regular. The
$\calf$-maximality of $f$ now guarantees that $v \leq f$ on $V$.

Since our sets $V$ form a base for the fine topology on $U$, and
$f$ is finite and finely continuous, the function $f$ must be
finely superharmonic. Since $f$ is also finely subharmonic, it
must be finely harmonic.

Finally, since being finely harmonic is an $f$-local property, we
can conclude that being $\calf$-maximal must also be an $f$-local
property. As $f$-local and $\calf$-local coincide in this setting,
$f$ is $\calf$-maximal on $U$ if and only if $f$ is
$\calf$-locally $\calf$-maximal on $U$.
\end{proof}

Let $U$ be a fine domain in $\RR^n$. A function $f \colon U
\rightarrow \RR_{+}\cup \{+ \infty\}$ is said to be
\emph{invariant} if there exists a sequence $(V_j)$ of finely open
sets such that $\tilde{V_j} \subseteq r(U)$ for all $j$, $\bigcup_{j
\in \NN} V_j =U$ and $\widehat{R}^{U \setminus V_j}_f=f$ for all
$j$. Here, $r(U)$ denotes the least regular finely open set containing $U$. See \cite{F5,EF} for more details. The invariant functions form
the fine analog of the non-negative harmonic functions in the
Euclidean case. A non-negative finely harmonic function on $U$ is
invariant, but the converse does not hold in general when $n>2$. However, an
invariant function $h$ on $U$ will be finely harmonic on the fine
open set $\{h < + \infty\}$ by \cite[Theorem 10.2]{F1}.

Let $G_U$ be the fine Green kernel of $U$ (see \cite{F4,F5}). Then
any non-negative finely superharmonic function $u$ can be uniquely
written as $u=G_U^{\mu} +h$, where $\mu \geq 0$ is a Borel measure
on $U$ and $h$ is an invariant function on $U$. So if
$u=G_U^{\nu_1} +h_1=G_U^{\nu_2}+h_2$ for Borel measures $\nu_1
\geq 0$ and $\nu_2 \geq 0$ on $U$ and invariant functions $h_1$
and $h_2$, then $\nu_1=\nu_2$ and $h_1=h_2$.

\begin{prop}\label{invariant geeft locmax}
Let $U \subseteq \CC$ be an $f$-open set, and $f \leq 0$ a finely
subharmonic function on $U$ such that $-f$ is invariant. Then $f$
is $\calf$-locally $\calf$-maximal on $U$.
\end{prop}

\begin{proof}
Working in $\CC$, the fine and plurifine topologies are equal.
Since we are looking for an $f$-local result, we can cover $U$ by
bounded regular $f$-open subsets $U_i$ and prove $\calf$-local
$\calf$-maximality on each $U_i$. By \cite[Theorem 3.22]{EF}, the
function $-f|_{U_i}$ will still be invariant. Therefore, we can
assume w.l.o.g. that $U$ is bounded and regular.

Since $-f$ is invariant, we can find an increasing sequence
$(V_j)_{j \in \mathbb{N}}$ of $f$-open sets such that
$\bigcup_{j=1}^{\infty} V_j = U$, $\tilde{V_j} \subseteq U$ for
all $j \in \mathbb{N}$, and $\widehat{R}^{U \setminus
V_j}_{-f}=-f$. But using \cite[Theorem 11.12]{F1}, we can write
$\widehat{R}^{U \setminus V_j}_{-f}= \lim_{k \rightarrow \infty}
\widehat{R}^{U \setminus V_j}_{\min(-f,k)}$. As $\min(-f,k)$ is
finite and finely superharmonic, the function $\widehat{R}^{U
\setminus V_j}_{\min(-f,k)}$ is finely harmonic on $V_j$ by
\cite[Theorem 10.2]{F1}. As a consequence, the functions
$\widehat{R}^{U \setminus V_j}_{\min(-f,k)}$ and $-\widehat{R}^{U
\setminus V_j}_{\min(-f,k)}$ will both be $\calf$-maximal on $V_j$
by Proposition \ref{prop2.9} above. The function $f=-
\widehat{R}^{U \setminus V_j}_{-f}$ is therefore the limit of a
decreasing sequence of $\calf$-maximal functions on $V_j$, hence
$\calf$-maximal on $V_j$ by Proposition \ref{prop2.2echt}. Since
the $V_j$ cover $U$, $f$ will be $\calf$-locally $\calf$-maximal
on $U$.
\end{proof}

\section{Maximalized $\calf$-plurisubharmonic functions}

As mentioned in Example \ref{example2.4}, we can find examples of
$\calf$-maximal functions by studying the maximalized function
$$f_A^*(z)=(\sup\{u(z): \ u\in \FPSH_-(U) \textrm{ and } u\le f \
\ {\rm in} \ A\})^*$$ of a function $f \in \FPSH_{-}(U)$, where
$U$ is an $\calf$-open subset of $\CC^n$ and $A \subseteq U$.

\begin{prop}\label{prop3.6}
Let $U$ be a bounded $\calf$-open subset of $\CC^n$, $f\in
\FPSH_-(U)$ and $V$ an ${\cal F}$-open subset of $U$, then
$f_{U\setminus V}^*$ is ${\cal F}$-maximal on $V$.
\end{prop}
\begin{proof}
Let $G$ be an $\calf$-open set such that ${\overline G}\subset V$
and let $v\in \FPSH(G)$ be bounded from above on $G$,  ${\cal
F}$-upper semicontinuous on $\overline{G}^{\calf}$, and satisfy
$v\le f_{U\setminus V}^*$ on ${\partial_\calf G}$.
 Let us put
\begin{equation*}
w =
\begin{cases}
{\max} (f_{U\setminus V}^*,v)&\text{on $G$,}\\
f_{U\setminus V}^*& \text{on $U\setminus G$}.
\end{cases}
\end{equation*}
By Proposition \ref{prop2.2} $w\in\FPSH(U)$.

We have $w=f_{U\setminus V}^*$ on $U\setminus V$ because
$U\setminus V\subset U\setminus G$. Since $f$ is
$\calf$-plurisubharmonic on $U$, one has $f_{U\setminus V}=f$ on
$U\setminus V$. On the other hand, it follows from \cite[Theorem
3.9]{EFW} that $f_{U\setminus V}^*=f_{U\setminus V}$ outside of a
pluripolar set $A\subset U$. As $U$ is bounded, we can find an
$\calf$-plurisubharmonic function $\psi<0$ on $U$ such that
$A\subset \{\psi=-\infty\}$. Then for any $\alpha>0$, we have
$w+\alpha \psi \le f$ on $U\setminus V$, hence $w+\alpha \psi \le
f_{U\setminus V}^*$ on $U$.  By letting $\alpha$ tend to $0$, we
obtain $w\le f_{U\setminus V}^*$ on $U\setminus \{\psi=-\infty\}$.
As we know that $w\ge f_{U\setminus V}^*$, the functions $w$ and
$f_{U\setminus V}^*$ can only differ on a pluripolar set. However,
as both are $\calf$-continuous, and pluripolar sets have empty
$\calf$-interiors, these functions must be identical on $U$. It
follows then that $v\le  f_{U\setminus V}^*$. Hence the
restriction of $f_{U\setminus V}^*$ to $V$ is ${\cal F}$-maximal.
\end{proof}

In the above proposition we have supposed that $U$ is bounded,
which guarantees for any pluripolar subset $A$ of $U$ the
existence of a function $\psi\in \FPSH_-(U)$ such that $A\subset
\{\psi=-\infty\}$. For a general $U$ we have:

\begin{prop}\label{prop3.7}
Let $U$ be an $\calf$-open subset of $\CC^n$, $f\in \FPSH_{-}(U)$
and $C\subseteq U$ be an ${\cal F}$-closed subset relative to $U$.
Then $f_{U\setminus C}^*$ is ${\cal F}$-maximal on the
$\calf$-interior $V$ of $C$.
\end{prop}

\begin{proof}
Let $G$ be an $\calf$-open set such that ${\overline G}\subset V$
and let $v\in \FPSH(G)$ be bounded from above on $G$,  extend
${\cal F}$-upper semicontinuously to $\partial_{\cal F} G$, and
satisfy $v\le f_{U\setminus \overline{V}^{\calf}}^*$ on
${\partial_\calf G}$.
 Let us put
\begin{equation*}
w =
\begin{cases}
{\max} (f_{U\setminus C}^*,v)&\text{on $G$,}\\
f_{U\setminus C}^*& \text{on $U\setminus G$}.
\end{cases}
\end{equation*}
By Proposition \ref{prop2.2}, $w\in\FPSH(U)$.

As $G\subset V \subset C$, we have $w=f_{U\setminus C}^*$ on
$U\setminus C$. But note that $f_{U\setminus C}=f$ on $U\setminus
C$, and as $U\setminus C$ is $\calf$-open and $f$ is $\calf$-upper
semicontinuous, we find that $f_{U\setminus C}^*=f$ on $U\setminus
C$. Therefore, we have $w=f$ on $U\setminus C$, which implies that
$w\le f_{U\setminus C}^*$ on $U$. It follows that $v\le
f_{U\setminus C}^*$. Hence the restriction of $f_{U\setminus C}^*$
to $V$ is ${\cal F}$-maximal.
\end{proof}

In particular, if $V\subseteq U$ is an $\calf$-open set, and $f
\in \FPSH_{-}(U)$, then $f_{U\setminus \overline{V}^{\calf}}^*$
and $f_{U\setminus \overline{V}}^*$ are ${\calf}$-maximal on $V$.

The next proposition translates this back to information about
open sets:

\begin{prop}\label{prop3.8}
Let $U$ be an $\calf$-open subset of $\CC^n$, $f \in \FPSH_{-}(U)$
and $C \subseteq U$ a subset that is relatively $\calf$-closed
with respect to $U$, with $\calf$-interior $V$. Then $f^*_{U
\setminus V}=f^*_{U \setminus C}$ and is therefore $\calf$-maximal
on $V$.
\end{prop}

\begin{proof}
Since $V \subseteq C$ we immediately find that $f_{U \setminus V}
\leq f_{U \setminus C}$. Now suppose that $u \in \FPSH_{-}(U)$ and
$u \leq f$ on $U \setminus C$. Since $u$ and $f$ are both
$\calf$-continuous, the set $\{u>f\}$ will be $\calf$-open and
contained in $C$. As $V$ is the $\calf$-interior of $C$, this
means that $\{u>f\} \subseteq V$, and so $u \leq f$ on $U
\setminus V$. This proves that $f_{U \setminus V}=f_{U \setminus
C}$, hence $f^*_{U \setminus V}=f^*_{U \setminus C}$, which is
$\calf$-maximal on $V$ by Proposition \ref{prop3.7}.
\end{proof}

For general $\calf$-open $V \subseteq U$, however, it is not true
that $f^*_{U \setminus V}=f^*_{U \setminus \overline{V}^{\calf}}$:

\begin{example}
Let $U=B(0,3) \subseteq \CC$, $V=B(0,2) \setminus C(0,1)$, and
define $f \in \FPSH_{-}(U)$ by $f(z)=|z|^2-10$. Then $U$ and $V$
are both Euclidean open, and
$\overline{V}^{\calf}=\overline{B}(0,2)$. Since $U$ is Euclidean
open in $\CC$, subharmonic functions and finely subharmonic
functions on $U$ are the same, by \cite[Proposition 2.14]{EFW}. So
we can use the maximum principle to see that $f_{U \setminus
\overline{B}(0,2)}^*=f_{U \setminus
\overline{B}(0,2)}=\max(|z|^2-10,-6)$.

On the other hand, $f_{U \setminus V} \leq -9$ on $C(0,1)$, and by
\cite[Theorem 3.9]{EFW} the set $\{f^*_{U\setminus
V}>f_{U\setminus V}\}$ is pluripolar. Since the functions $f_{U
\setminus V}$ and $f^*_{U \setminus \overline{B}(0,2)}$ are
unequal on $C(0,1)$, and $C(0,1)$ is not a pluripolar set, the
functions $f^*_{U \setminus V}$ and $f^*_{U \setminus
\overline{B}(0,2)}$ cannot be equal.
\end{example}

In the case where $U$ is bounded, the argument in the proof of
Proposition \ref{prop3.6} can be applied to get one more equality:

\begin{prop}
Let $U \subseteq \CC^n$ be a bounded $\calf$-open set, and let $A
\subseteq B \subseteq U$ be subsets such that $B \setminus A$ is
pluripolar. Then for any $f \in \FPSH_{-}(U)$ we have $f^*_{U
\setminus A}=f^*_{U \setminus B}$.
\end{prop}

\begin{proof}
Since $A \subseteq B$ we immediately find that $f_{U \setminus A}
\leq f_{U \setminus B}$.

On the other hand, suppose that $u \in \FPSH(U)$ such that $u \leq
f$ on $U \setminus B$. Since $B \setminus A$ is pluripolar, we can
find a plurisubharmonic function $\psi \in \PSH(\Omega)$ such that
$B \setminus A \subseteq \{\psi = - \infty \}$. Since $U$ is
bounded, we may assume that $\psi <0$ on $U$. Now for any $\alpha
>0$ we have $u+\alpha \psi \in \FPSH_{-}(U)$ and $u+\alpha \psi
\leq f$ on $U \setminus A$. Letting $\alpha$ tend to zero, we see
that $f_{U \setminus A}=f_{U \setminus B}$ outside the pluripolar
set $\{\psi = -\infty\}$. By \cite[Theorem 3.9]{EFW}, we know that
$\{f^*_{U \setminus A}>f_{U \setminus A}\}$ and $\{f^*_{U
\setminus B}>f_{U \setminus B}\}$ are also pluripolar. As a union
of three pluripolar sets is pluripolar, we see that $f^*_{U
\setminus A}$ and $f^*_{U \setminus B}$ are equal outside a
pluripolar set. But both these functions are $\calf$-continuous,
and a pluripolar set has an empty $\calf$-interior, so the
functions must be equal on all of $U$.
\end{proof}

\section{Maximal ${\cal F}$-plurisubharmonic and the Monge-Amp\`ere operator}

In view of \cite[Corollary 3.4]{BT} the maximalized function
$u_{\Omega \setminus U}^*$, where $U$ is $\calf$-open, $\Omega$ is
Euclidean open, $U \subseteq \Omega$, and $u \in \PSH_{-}(\Omega)$
is locally bounded, satisfies $(dd^c u_{\Omega \setminus
U}^*)^n|{U}=0$.

This leads one to ask whether one has $(dd^c u)^n=0$ for all
finite $u\in\FMPSH(U)$? In the next theorem we shall prove that
the answer is yes when $u$ is the restriction to $U$ of a locally
bounded plurisubharmonic function on a Euclidean domain containing
$U$. In Theorem \ref{thm3.12} we will show that the answer is yes
in the general case.

\begin{theorem}\label{thm2.5}
Let $U$ be an $\calf$-open subset of a Euclidean open set
$\Omega$. Let $u\in\psh(\Omega)$ be locally bounded. If $u$ is
$\calf$-maximal in $U$, then $(dd^cu)^n|U=0$.
\end{theorem}

\begin{proof}

We can cover $U$ with sets $V_i$ that are $\calf$-open and satisfy
$\overline{V_i} \subseteq U$ and $\overline{V_i} \subseteq W_i
\subseteq \overline{W_i} \subseteq \Omega$ for some bounded
Euclidean open set $W_i$. By subtracting a constant if necessary,
we may temporarily assume that $u<0$ on $\overline{W_i}$, which
allows us to look at the maximalized function $$^{W_i} u_{W_i
\setminus V_i}^*=(\sup\{w \in \PSH_{-}(W_i): w \leq u \text{ in }
W_i \setminus V_i\})^*.$$

By \cite[Corollary 3.4]{BT} we have $(dd^c (^{W_i} u_{W_i\setminus
V_i}^*))^n|V_i=0$. Because $u$ is $\calf$-maximal in $U$, we have
$u={^{W_i}u_{W_i \setminus V_i}^*}$ on $W_i$, so $(dd^c
u)^n|{V_i}=0$. The result follows from the quasi-Lindel\"of
property of the plurifine topology and the fact that the
Monge-Amp\`ere measure of a locally bounded plurisubharmonic
function does not charge pluripolar sets.
\end{proof}

We denote by $\QB(\CC^n)$ the $\sigma$-algebra on $\CC^n$
generated by the Borel sets and the pluripolar subsets of $\CC^n$.
If $U$ is an ${\cal F}$-open set in $\CC^n$, we denote by $\QB(U)$
the trace of $\QB(\CC^n)$ on $U$. The elements of $\QB(U)$ are
called quasi-borelian subsets of $U$.

\begin{definition}\label{def3.1}
Let $U$ be an ${\cal F}$-open set in $\CC^n$ and let $(\mu_j)$ and
$\mu$ be measures on $(U,\QB(U))$ that give measure zero to
pluripolar sets. We say that $(\mu_j)$ converges ${\cal
F}$-locally vaguely to $\mu$ if for any $z\in U$ there exists an
${\cal F}$-open $V$ such that $z\in V\subset U$ and
\[
\lim_{j\to +\infty}\int\varphi d\mu_j=\int\varphi d\mu
\]
for every bounded ${\cal F}$-continuous function $\varphi$ with compact
support on $V$.
\end{definition}

\begin{remark}\label{remark3.2}
Let $(\mu_j)$, $\mu$ and $\nu$ be measures on $(U,\QB(U))$ that
give measure zero to pluripolar sets. Suppose that $(\mu_j)$
converges ${\cal F}$-locally vaguely to both $\mu$ and $\nu$. We
will show that $\mu=\nu$.

Indeed, let $z \in U$ and pick $V_z$ such that $\int \varphi
d\mu=\lim_{j \to +\infty} \int \varphi d\mu_j=\int \varphi d\nu$
for every bounded $\calf$-continuous function $\varphi$ with
compact support on $V_z$. Using \cite[Theorem 2.3]{BT} we can find
$r>0$ and $\varphi_1 \in \PSH(B(z,r))$ such that $z \in \{w \in
B(z,r):\varphi_1(w)>0\} \subseteq V_z$. By adapting this
$\varphi_1$, adding a continuous bump function, and adapting a bit
more, we can find an $\calf$-continuous function $\varphi$ on
$\CC^n$ with compact support within $V_z$ such that $0 \leq
\varphi \leq 1$ and $\varphi(w) \equiv 1$ on an $\calf$-open
neighbourhood $W_z$ of $z$.

Using the quasi-Lindel\"{o}f property of the plurifine topology,
we can find a sequence $(z_k)$ of points in $U$ such that
$U=\bigl( \bigcup_{k=1}^\infty W_{z_k} \bigr) \cup K$ where $K$ is
pluripolar. Define $\psi_1=\varphi_{z_1}$,
$\psi_2=\min(\varphi_{z_2},1-\psi_1)$, \ldots,
$\psi_k=\min(\varphi_{z_k},1-(\psi_1+ \ldots +\psi_{k-1}))$,
\ldots. These functions $\psi_k$ form a partition of unity for
$U\setminus K$, and $\mu(K)=\nu(K)=0$. If $\eta$ is a bounded
$\calf$-continuous function $\geq 0$ on $U$, we find that
\begin{align*}
\int_U \eta d\mu&=\int_{U \setminus K} \eta d\mu = \Sigma_{k=1}^{\infty} \int_{U \setminus K} \psi_k \eta d\mu\\
&=\Sigma_{k=1}^{\infty} \int_{U \setminus K} \psi_k \eta d\nu=\int_{U \setminus K} \eta d\nu\\
&=\int_{U} \eta d\nu.\\
\end{align*}
This implies that $\mu=\nu$.
\end{remark}

Recall the following result of Bedford and Taylor, \cite{BT}.

\begin{theorem}[Bedford and Taylor, \cite{BT}]\label{thm3.3}
Let $(u_j^1)$,...,$(u_j^n)$ be monotone sequences of bounded
plurisubharmonic functions on a Euclidean domain $\Omega$,
converging respectively to bounded plurisubharmonic functions
$u_1,...,u_n$. Then for every bounded quasi ${\cal F}$-continuous
function $\varphi$ on  $\Omega$, we have
\[
\lim_{j\to +\infty}\int\varphi dd^cu_j^1\wedge...\wedge dd^cu_j^n
=\int\varphi dd^cu_1\wedge...dd^c u_n.
\]
\end{theorem}

\begin{theorem}\label{thm3.4}
Let $(f_j)$ be a monotone sequence of finite,
$\calf$-plurisubharmonic functions on an $\calf$-domain $U$ in
$\CC^n$ that converge to a finite $f\in\FPSH(U)$. Then the
sequence of measures $(dd^cf_j)^n$ converges ${\cal F}$-locally
vaguely to $(dd^cf)^n$.
\end{theorem}

\begin{proof}
As $(f_j)$ is monotone, the function $f$ and all of the functions
$f_j$ are ${\cal F}$-locally uniformly bounded on $U$. By
\cite[Theorem 2.4]{EFW}, for every $z\in U$ there exists an open
ball $B_z=B(z,r_z)$ in ${\CC}^n$ and an ${\cal F}$-open
neighborhood $V_z\subset B_z$ of $z$, and functions $u_j^z$, $j\in
{\NN}$, $u$ and $\Phi$ that are plurisubharmonic and uniformly
bounded on $B_z$ such that $f_j=u_j-\Phi$, $f=u-\Phi$ on $V_z$ and
the sequence $(u_j)$ is monotone. Recall from \cite{EW} that
$$(dd^cf_j)^n=\sum_{p=0}^n (-1)^p {{n}\choose{p}} (dd^c u_j)^p
\wedge (dd^c \Phi)^{n-p} \quad \textrm{on } V_z.$$ Let $\varphi$
be a bounded ${\cal F}$-continuous function on $V_z$ that equals
$0$ outside of a compact $K\subset V_z$. Then according to Theorem
\ref{thm3.3}
\begin{equation*}
\begin{split}
\lim_{j\to +\infty}\int\varphi (dd^cf_j)^n=
\lim_{j\to +\infty}\sum_{p=1}^n (-1)^p  {{n}\choose{p}} \int\varphi(dd^c(u_j)^p\wedge(dd^c\Phi))^{n-p}\\
=\sum_{p=1}^n(-1)^p  {{n}\choose{p}}\int\varphi(dd^cu)^p\wedge (dd^c\Phi))^{n-p}
=\int\varphi (dd^cf)^n.
\end{split}
\end{equation*}
Hence the sequence of measures $(dd^cf_j)^n$ converges ${\cal
F}$-locally vaguely to $(dd^cf)^n$.
\end{proof}

\begin{cor}\label{cor3.5}
Let $(u_j)$ be a monotone sequence of finite,
$\calf$-plurisubharmonic functions on an $\calf$-domain $U$ in
$\CC^n$ that converge to a finite $u\in\FPSH(U)$. Then
$(dd^cu)^n=0$ if $(dd^cu_j)^n=0$ for all $j$.
\end{cor}

\begin{proof}
Applying Theorem \ref{thm3.4}, we see that $(dd^c u_j)^n$
converges $\calf$-locally vaguely to $(dd^c u)^n$. However, by
Definition \ref{def3.1} $(dd^c u_j)^n$ converges $\calf$-locally
vaguely to 0 as well. By Remark \ref{remark3.2}, this means that
$(dd^c u)^n=0$.
\end{proof}

\begin{lemma}\label{lemma3.11}
Let $f$ be an $\calf$-plurisubharmonic $\calf$-maximal function
$\le 0$ on an $\calf$-domain $U$ in $\CC^n$. Then for any ${\cal
F}$-open subset $V$ of $U$ such that $\overline{V}\subset U$ we
have $f_{U\setminus V}^*=f_{U\setminus \overline{V}^\calf}^*=f.$
\end{lemma}

\begin{proof}
The equality $f_{U\setminus V}^*=f$ is an immediate consequence of
the definition of ${\calf}$-maximal functions.

If we let $W$ be the $\calf$-interior of $\overline{V}^{\calf}$,
then Proposition \ref{prop3.8} tells us that $f^*_{U \setminus
\overline{V}^{\calf}}=f^*_{U \setminus W}$. And as $W$ is
$\calf$-open and $\overline{W}=\overline{V} \subseteq U$, we have
$f^*_{U \setminus W}=f$ since $f$ is $\calf$-maximal.
\end{proof}

\begin{theorem}\label{thm3.12}
Let $f$ be a finite ${\cal F}$-maximal $\calf$-plurisubharmonic
function on an $\calf$-domain $U$ in $\CC^n$. Then we have $(dd^c
f)^n=0$.
\end{theorem}

\begin{proof}
By \cite[Theorem 2.14]{EW}, there exists an ${\cal F}$-closed
pluripolar set $E$ such that $f$ is C-strongly
$\calf$-plurisubharmonic on $U\setminus E$. Let $z$ be an element
of $U\setminus E$, then there is an ${\cal F}$-neighbourhood
$V\subset {\overline V}\subset U\setminus E$ of $z$ and a sequence
$(f_j)$ of plurisubharmonic functions on neighbourhoods of
$\overline{V}$ that converge uniformly to $f$ on $\overline{V}$.
Since $f$ is $\calf$-continuous, we can shrink $V$ a bit, to
ensure that $f$ is bounded on $\overline{V}$. By subtracting a
constant from $f$ and all $f_j$, we may assume that $f<0$ on
$\overline{V}$.

For each $j$, $f_j$ is plurisubharmonic on some Euclidean open set
$O_j$ containing ${\overline V}$. We may assume that the sequence
$(O_j)$ is decreasing, and that $f_{j+1}\geq f_j$ on $O_{j+1}$
(use the uniform convergence to adjust the sequence $(f_j)$ to
become increasing on $\overline{V}$, and then replace $f_{j+1}$ by
$\max(f_j,f_{j+1})$ on $O_{j+1}$). Now let $m$ be a lower bound
for $f$ on $\overline{V}$, and replace each $f_j$ by
$\max(f_j,m)$, to make these (upper semicontinuous) functions
locally bounded. Let $\epsilon >0$, then we can find $j_0$ such
that $f_j \le f \le f_j+\epsilon$ on $V$ for all $j\ge j_0$. Let
$O$ be an ${\cal F}$-open set such that $O\subset {\overline
O}\subset V$, then on $V$ we have for all $j \geq j_0$ (using
Lemma \ref{lemma3.11})
$$f-\epsilon \leq f_j \le {^{O_j}(f_j)}^*_{O_j\setminus O} \le {^{O_j}(f_j)}^*_{V \setminus O}
\le {^V(f_j)}^*_{V\setminus O} \le {^Vf}^*_{V\setminus O}=f,$$
which implies that the sequence of restrictions to $V$ of the
functions ${^{O_j}}(f_j)^*_{O_j\setminus O}$ converges uniformly
to $f$. However, we know that $(dd^c({^{O_j}(f_j)}^*_{O_j\setminus
O}))^n=0$ on $O$ by \cite[cor 3.4]{BT}. Now note that the sequence
${^{O_j}}(f_j)^*_{O_j\setminus O}$ is increasing on $V$:
${^{O_{j}}}(f_j)^*_{O_{j}\setminus O} \le {^{O_j}(f_j)}^*_{O_{j+1}
\setminus O} \leq {^{O_{j+1}}}(f_j)^*_{O_{j+1}\setminus O} \leq
{^{O_{j+1}}}(f_{j+1})^*_{O_{j+1}\setminus O}$. Hence we can apply
Corollary \ref{cor3.5} to see that $(dd^cf)^n=0$ on $O$. As we can
cover all of $U \setminus E$ by such sets $O$, we can now see that
$(dd^cf)^n=0$, using the definition of the Monge-Amp\`{e}re
operator for $\calf$-psh functions in \cite[def 4.5]{EW}.
\end{proof}

\begin{cor}\label{cor3.13}
Let $f$ be a finite ${\cal F}$-plurisubharmonic function $\le 0$
on an $\calf$-domain $U$ in $\CC^n$, and  $V$ an ${\cal F}$-open
subset of $U$, then we have $(dd^cf_{U\setminus
\overline{V}^{\calf}}^*)^n=0$ and $(dd^cf_{U\setminus
\overline{V}}^*)^n=0$ on $V$.

Also, if $U$ is bounded, or if $V$ equals the $\calf$-interior of
some $\calf$-closed set, we have $(dd^cf_{U\setminus V}^*)^n=0$ on
$V$.
\end{cor}

\begin{proof} The result follows immediately from the above theorem
and the Propositions \ref{prop3.6}, \ref{prop3.7} and \ref{prop3.8}.
\end{proof}

\begin{question}\label{question3.14}
Do we have the converse of the above theorem? The result is
well-known in the case of an Euclidean open $\Omega$ for locally
bounded plurisubharmonic functions on $\Omega$. See for example
\cite[Corollary 3.7.6]{K}.
\end{question}

The following theorem gives an affirmative answer to this question
when $f$ is the restriction to an $\calf$-open subset $U$ of a
locally bounded plurisubharmonic function on a Euclidean domain
$\Omega$.

\begin{theorem}\label{thm3.15}
Let $f$ be a locally bounded plurisubharmonic function on a
Euclidean domain $\Omega$. If $(dd^cf)^n=0$ on an $\calf$-open
subset $U$ of $\Omega$, then the restriction of $f$ to $U$ is
${\cal F}$-maximal.
\end{theorem}

\begin{proof} Let $G$ be an ${\cal F}$-open set such that
$G\subset {\overline G}\subset U$ and $v$ an upper bounded  ${\cal
F}$-psh function on $G$ such that $\calf$-$\limsup_{z\in G, z\to
\zeta}v(z)\le u(\zeta)$ for any $\zeta\in {\partial_{\cal F}}G$.
Let us put
$$g=
\left\{\begin{array}{cc}
\max(f,v) & {\rm on} \ G\\
f & {\rm on} \ \Omega \setminus G.
\end{array}
\right.$$ Then by Proposition 2.2 and \cite[Proposition
2.14]{EFW}, $g$ is a locally bounded psh function on $\Omega$.
Moreover we have $(dd^cf)^n=(dd^cg)^n$ on $\Omega\setminus
{\overline G}$ and $(dd^cf)^n=0$ on $U$, hence $(dd^cf)^n\le
(dd^cg)^n$. Since $f=g$ on $\Omega\setminus {\overline G}$, it
follows from the comparison principle \cite[Corollary 4.5]{K} that
$g\le f$, and therefore $f=g$. We then deduce that $v\le f$. So we
have proved that  $u$ is ${\cal F}$-maximal.
\end{proof}

Now let us recall the following result
of B{\l}ocki \cite{BL1}:\\

\begin{theorem}[B{\l}ocki,\cite{BL1}]\label{thm3.16}
Let $\Omega$ be a bounded Euclidean domain of ${\bf C}^n$ and let
$u,h\in \PSH\cap L^{\infty}_{loc}(\Omega)$ such that $u\le h$ and
$\lim_{z\to
\partial \Omega}(h(z)-u(z))=0$. Then for $R=\min\{r>0 :\Omega \subset
B(z_0,r) \textrm{ for some }z_0 \in \mathbb{C}^n\}$ we have
$$||h-u||_{L^n(\Omega)}\le \frac{R^2}{4}(\int_{\Omega}(dd^cu)^n)^{\frac{1}{n}}.$$
where $||h-u||_{L^n(\Omega)}=(\int_{\Omega} |h-u|^n
d\lambda)^{\frac{1}{n}}$.
\end{theorem}

\begin{lemma}\label{lemma3.18}
Let $U$ be an ${\cal F}$-open subset of $\CC^n$, $f\in
\FPSH_{-}(U)$ and $(O_j)$ a decreasing sequence of Euclidean open
sets such that $U\subset \cap_jO_j$. Let $(f_j)$ be a sequence of
functions such that $f_j\in \PSH_{-}(O_j)$ and $f_{j+1}\le f_j$ on
$O_{j+1}$ for all $j$, and such that the sequence $(f_j|_U)$
converges uniformly to $f$ on $U$. Let $G$ be ${\cal F}$-open such
that $\overline{G} \subset U$, and let $(\omega_j)$ be a
decreasing sequence of $\calf$-open sets such that ${\overline
\omega_j} \subset O_j$ for all $j$ and $\cap_j{\overline
\omega_j}={\overline G}$. Then the sequence of the restrictions of
${^{O_j}(f_j)}^*_{O_j\setminus{\overline \omega_j}}$ to $U$ is
decreasing to ${^Uf}_{U\setminus {\overline G}}^*$.
\end{lemma}

\begin{proof}
For any $j$ we have ${^{O_j}(f_j)}_{O_j\setminus{\overline
\omega_j}} = f_j$ on $O_j \setminus \overline{\omega}_j$, and as
$f_j$ is upper semicontinuous, and $O_j \setminus
\overline{\omega}_j$ is open, this means that
${^{O_j}(f_j)}^*_{O_j\setminus{\overline \omega_j}} = f_j$ on $O_j
\setminus \overline{\omega}_j$.

On $O_{j+1} \setminus \overline{\omega}_{j+1}$ we therefore find
that ${^{O_{j+1}}(f_{j+1})}^*_{O_{j+1}\setminus{\overline
\omega_{j+1}}}=f_{j+1} \leq f_j$, so
$\{{^{O_{j+1}}(f_{j+1})}^*_{O_{j+1}\setminus{\overline
\omega_{j+1}}} > f_j\}$ is an $\calf$-open set contained in
$\overline{\omega}_{j+1}\subseteq \overline{\omega}_{j} \subseteq
O_j$. Applying Proposition \ref{prop2.2} allows us to glue these
functions together to an $\calf$-plurisubharmonic function
\begin{equation*}
h_j =
\begin{cases}
\max(f_j,{^{O_{j+1}}(f_{j+1})}^*_{O_{j+1}\setminus{\overline \omega_{j+1}}})&\text{on $\{{^{O_{j+1}}(f_{j+1})}^*_{O_{j+1}\setminus{\overline \omega_{j+1}}} > f_j\}$,}\\
f_j& \text{on $\ O_j \setminus \{{^{O_{j+1}}(f_{j+1})}^*_{O_{j+1}\setminus{\overline \omega_{j+1}}} > f_j\}$.}
\end{cases}
\end{equation*}
As $h_j$ is $\calf$-plurisubharmonic and $h_j = f_j$ on $O_j
\setminus \overline{\omega}_j$, we find that $h_j \leq
{^{O_j}(f_j)}_{O_j\setminus{\overline \omega_j}} \leq
{^{O_j}(f_j)}^*_{O_j\setminus{\overline \omega_j}}$. Hence
${^{O_{j+1}}(f_{j+1})}^*_{O_{j+1}\setminus{\overline
\omega_{j+1}}} \leq {^{O_j}(f_j)}^*_{O_j\setminus{\overline
\omega_j}}$ on $O_{j+1}$.

Since ${^{O_j}(f_j)}^*_{O_j\setminus{\overline \omega_j}} = f_j$
on $O_j \setminus \overline{\omega}_j$, this implies that $\lim_{k
\to + \infty} {^{O_k}(f_k)}^*_{O_k\setminus{\overline \omega_k}}
\leq f_j$ on $U \setminus \overline{\omega}_j$ for all $j$. As the
$f_j$ decrease to $f$, and $\cap_j{\overline \omega_j}={\overline
G}$, we find that $\lim_{k \to + \infty}
{^{O_k}(f_k)}^*_{O_k\setminus{\overline \omega_k}} \leq f$ on $U
\setminus \overline{G}$.

Combining this with the fact that
${^{O_j}(f_j)}^*_{O_j\setminus{\overline \omega_j}} \geq f_j \geq
f$ on $U$ for all $j$, we get
$$\lim_j{^{O_j}(f_j)}^*_{O_j\setminus{\overline \omega_j}}=f \textrm{ in } U\setminus {\overline G}.$$
As the limit of a monotonically decreasing sequence of
$\calf$-plurisubharmonic functions is again
$\calf$-plurisubharmonic (see \cite[p. 84]{F1}), we now find that
$$\lim_j{^{O_j}(f_j)}^*_{O_j\setminus {\overline\omega_j}}\le {^Uf}_{U\setminus {\overline G}}^* \textrm{ in } U.$$

Before proving the inverse inequality, note that the conclusion of
the lemma only depends on $\overline{G}$ rather than $G$ itself.
Let $H$ be the $\calf$-interior of $\overline{G}$. Now $G
\subseteq H \subseteq \overline{G}$, and hence
$\overline{H}=\overline{G}$. So without loss of generality, we may
replace $G$ by $H$ and assume that $G$ equals the $\calf$-interior
of $\overline{G}$.

Using this new assumption, the set $\overline{ U \setminus
\overline{G}}^{\calf}$ has an $\calf$-open complement, whose
intersection with $U$ is contained in $\overline{G}$, hence in
$G$. As $\overline {G} \subseteq U$ and $G$ is $\calf$-open, it
follows that $\partial_{\calf}G \subseteq U \setminus G \subseteq
\overline{ U \setminus \overline{G}}^{\calf}$

\bigskip

We can now prove the inverse inequality. For all $j$ we have,
${^Uf}_{U\setminus {\overline G}} = f \leq f_j$ on $U \setminus
\overline{G}$, and by $\calf$-upper semicontinuity of $f_j$, this
means that ${^Uf}_{U\setminus {\overline G}}^* \leq f_j$ on $U
\setminus \overline{G}$ as well. Since both these functions are in
fact $\calf$-continuous, we can see that the inequality even holds
on $\overline{ U \setminus \overline{G}}^{\calf}$, and hence on
$\partial_{\calf}G$. Then define
$$g_j=
\left\{\begin{array}{cc}
\max(f_j,{^Uf}_{U\setminus {\overline G}}^*) & {\rm on} \ G\\
f_j & {\rm on} \ O_j \setminus G.
\end{array}
\right.$$
By Proposition \ref{prop2.2}, $g_j$ is
$\calf$-plurisubharmonic on $O_j$, and we can see that $g_j \leq
f_j$ on $O_j \setminus \overline{\omega}_j$ since $\overline{G}
\subseteq \overline{\omega}_j$. Therefore, $g_j \leq
{^{O_j}(f_j)}^*_{O_j\setminus {\overline \omega_j}}$, and so
${^Uf}_{U\setminus {\overline G}}^* \leq
{^{O_j}(f_j)}^*_{O_j\setminus {\overline \omega_j}}$ on $G$.
Combining this with the fact that ${^Uf}_{U\setminus {\overline
G}}^* \leq f_j$ on $\overline{U \setminus \overline{G}}^{\calf}$,
we find that ${^Uf}_{U\setminus {\overline G}}^* \leq
{^{O_j}(f_j)}^*_{O_j\setminus {\overline \omega_j}}$ on all of
$U$. Therefore, $$\lim_j{^{O_j}(f_j)}^*_{O_j\setminus{\overline
\omega_j}}\ge {^Uf}_{U\setminus {\overline G}}^*.$$
\end{proof}

Now we can prove the converse of Theorem \ref{thm3.12}:

\begin{theorem}\label{thm3.20}
Let  $f$ be a finite ${\cal F}$-plurisubharmonic function on an
$\calf$-domain $U\subseteq \CC^n$. If $(dd^cf)^n=0$ then there
exists an ${\cal F}$-closed pluripolar set $F\subset U$ such that
$f$ is ${\cal F}$-locally $\calf$-maximal on $U\setminus F$.
\end{theorem}

\begin{proof}
By \cite[Theorem 2.14]{EW} there exists an ${\cal F}$-closed
pluripolar subset $F$ of $U$ such that $f$ is C-strongly ${\cal
F}$-psh on $U \setminus F$. Let $z\in U\setminus F$, then there
exists an ${\cal F}$-open neighborhood $V_z$ such that $f=\lim
f_j$ uniformly on $\overline{V_z}$, where for all $j$, the
function $f_j$ is a continuous plurisubharmonic function on a
bounded Euclidean open set $O_j$ containing  ${\overline V_z}$. By
shrinking $V_z$, we may assume $f$ to be bounded on
$\overline{V_z}$. By shrinking the $O_j$, we may assume that
$O_{j+1}\subset O_j$, and using the uniform convergence on
$\overline{V_z}$ we can ensure that $f_{j+1} < f_j$ on $V_z$ for
all $j$. Replacing $O_{j+1}$ by $\{w \in
O_{j+1}:f_{j+1}(w)<f_j(w)\}$ will give us $f_{j+1} \le f_j$ on
$O_{j+1}$. Finally, as $f$ is bounded on $\overline{V_z}$ and the
$f_j$ are continuous and converge uniformly to $f$ on
$\overline{V_z}$, we can see that if we shrink the $O_j$ some
more, all $f_j$ will be uniformly bounded.

Now by Theorem \ref{thm3.4}, we find that $(dd^c f_j|_U)^n$
converges $\calf$-locally vaguely to $(dd^c f)^n$. So for every $w
\in V_z$, we can find an $\calf$-open $W$ such that $w \in W
\subseteq V_z$ and $\lim_{j\to +\infty}\int\varphi (dd^c
f_j|_U)^n=\int\varphi (dd^c f)^n$ for every bounded ${\cal
F}$-continuous function $\varphi$ with compact support on $W$. Now
construct a function $\varphi$ as in Remark \ref{remark3.2}:
$\varphi$ is $\calf$-continuous, has compact support within $W$,
satisfies $0 \leq \varphi \leq 1$ and satisfies $\varphi \equiv 1$
in some $\calf$-open neighbourhood $G$ of $w$.

Note that $G \subset \overline{G} \subset V_z$, and let
$(\omega_j)$ be a decreasing sequence of Euclidean open sets such
that for any integer $j$ we have $\overline{\omega_{j+1}}\subseteq
\omega_j$, ${\overline {\omega_j}}\subset O_j$ and
$\cap_j{\overline{\omega_j}}={\overline G}$.

Define $h_{j,k}={^{O_j}(f_j)}^*_{O_j\setminus {\overline
\omega_k}}$ for $k \geq j$, and $h_j={^{O_j}(f_j)}^*_{O_j
\setminus \overline{G}}$. Note that as the $f_j$ are uniformly
bounded, and $\overline{G} \subseteq \overline{\omega_j} \subseteq
O_j$ for all $j$, we can use the maximum principle to see that all
$h_{j,k}$ and all $h_j$ are uniformly bounded by those same
bounds. By applying Lemma \ref{lemma3.18} to the functions
$h_{j,k}$, where we let $k \rightarrow \infty$ and keep $j$ fixed,
we see that $(h_{j,k})_k$ decreases to $h_j$. By applying Lemma
\ref{lemma3.18} again (taking all $\omega_j$ in the lemma to be
$G$ this time), we see that $(h_j)$ decreases to ${^U f}^*_{U
\setminus \overline{G}}$.

Now by Theorem \ref{thm3.16}, if $R=\min\{r>0: O_1 \subseteq
B(z_0,r) \textrm{ for some } z_0 \in \CC^n\}$, we get
$$||h_{j,k+1}-f_j||_{L^n(G)}\leq ||h_{j,k+1}-f_j||_{L^n(\omega_k)} \le \frac{R^2}{4} \bigl(\int_{\omega_k}(dd^cf_j)^n\bigr)^{\frac{1}{n}}$$ for all $k\ge j$. Letting $k \rightarrow \infty$, we deduce that
$$||h_j-f_j||_{L^n(G)}\le \frac{R^2}{4} \bigl(\int_{\overline G}(dd^cf_j)^n\bigr)^{\frac{1}{n}}$$
for all $j$. By letting $j\to +\infty$, we obtain $||{^U
f^*_{U\setminus \overline{G}}}-f||_{L^n(G)}$ on the left-hand
side, by dominated convergence. On the right-hand side, the choice
of $G$ and the fact that $(dd^c f)^n=0$ imply that \[
\int_{\overline G}(dd^cf_j)^n \leq \int_W \varphi(dd^cf_j)^n
\rightarrow \int_W \varphi d(dd^c f)^n =0.\] Therefore, $f={^U
f^*_{U\setminus \overline{G}}}$ on $G$. By Proposition
\ref{prop3.8}, the function ${^U f^*_{U\setminus \overline{G}}}$
is $\calf$-maximal on $G$. This proves the $\calf$-maximality of
$f$ on $G$, and since all of $U \setminus F$ will be covered by
such sets $G$, $f$ is $\calf$-locally $\calf$-maximal on $U
\setminus F$.
\end{proof}

\begin{theorem}\label{cor3.21}
Let $U$ be an ${\cal F}$-domain in $\CC^n$ and let $f\in \FPSH(U)$
be finite. Then $(dd^cf)^n=0$ if and only if $f$ is ${\cal
F}$-locally $\calf$-maximal on the complement of an ${\cal
F}$-closed pluripolar subset of $U$.
\end{theorem}

\begin{proof}
One half of this theorem is just Theorem $\ref{thm3.20}$. For the
other implication, combining Theorem \ref{thm3.12} with the
quasi-Lindel\"{o}f property of the plurifine topology shows that
$U$ can be covered by countably many $\calf$-open sets $V$ where
$(dd^c f)^n|V=0$ and a pluripolar set $E$. By the definition of
the Monge-Amp\`{e}re operator for $\calf$-plurisubharmonic
functions in \cite[def 4.5]{EW}, this means that $(dd^cf)^n=0$.
\end{proof}

\begin{cor}
Let $U \subseteq \CC^n$ be an $\calf$-open set, and let $(u_j)$ be
a sequence of finite $\calf$-maximal functions in $\FPSH(U)$,
increasing to a finite $u \in \FPSH(U)$. Then $u$ is
$\calf$-locally $\calf$-maximal outside some closed pluripolar set
$F$.
\end{cor}

\begin{proof}
Since all $u_j$ are $\calf$-maximal, Theorem \ref{thm3.12} implies
that $(dd^c u_j)^n=0$ for all $j$. Applying Corollary
\ref{cor3.5}, we see that $(dd^c u)^n=0$. By Theorem
\ref{thm3.20}, we can now find a closed pluripolar set $F$ such
that $u$ is $\calf$-locally $\calf$-maximal on $U \setminus F$.
\end{proof}

\begin{question}\label{question3.23}
Is an ${\cal F}$-locally $\calf$-maximal function on an ${\cal
F}$-open set $U\subseteq \CC^n$ also $\calf$-maximal on $U$?
\end{question}

In the case of a Euclidean open set we have the following
positive answer:

\begin{prop}\label{prop3.24}
Let $f$ be a locally bounded plurisubharmonic  function on a
Euclidean open set $\Omega$. If $f$ is ${\cal F}$-locally
$\calf$-maximal on $\Omega$, then  $f$ is maximal.
\end{prop}

\begin{proof}
By Theorem \ref{cor3.21}  we have $(dd^cf)^n=0$, hence  $f$ is
maximal by \cite[Corollary 3.7.6]{K}.
\end{proof}

When working with $n=1$, we also have a positive answer for finite
functions on an $\calf$-open set $U\subseteq \CC^n$, as seen in
Proposition \ref{prop2.9}. However, this result is not valid when
the function is not finite, as seen in the example below.

\begin{example}
Let $\mu$ be a Borel measure on the unit disc $\mathbb{D}$ in $\CC$,
such that $\mu \geq 0$, $\textrm{supp}(\mu)=K$ for some compact
polar set $K \subseteq \mathbb{D}$, and $\mu$ has no atoms. Such a
measure is mentioned in \cite[Example 4.9]{ACCP} for example.

Let $u=-G^{\mu}_{\mathbb{D}}$, then $u$ will be a subharmonic
function on $\mathbb{D}$ such that $dd^c u = \mu$. Since
$\mu(\mathbb{D} \setminus K)=0$, the function $u$ is harmonic on
$\mathbb{D} \setminus K$, and hence maximal on $\mathbb{D}
\setminus K$. By Proposition \ref{prop2.3}, $u$ will be
$\calf$-maximal on $\mathbb{D} \setminus K$ as well.

Now let $z \in K$. Since $K$ is a polar set in $\CC$, each of its
points is finely isolated (see \cite[p. 58]{D}). So we can find an
$\calf$-open set $V \subseteq \mathbb{D}$ such that $V \cap K =
\{z\}$. Since $\mu$ has no atoms, we have $\mu(V)=0$. Now note
that $-u=G^{\mu}_{\mathbb{D}}=\int G_{\mathbb{D}}(\cdot,y)
d\mu(y)= \int G_V(\cdot,y)d\mu(y) + \int \widehat{R}^{\mathbb{D}
\setminus V}_{G_{\mathbb{D}}(\cdot,y)}d\mu(y)$. However, since
$\mu(V)=0$, the first integral will vanish. So for $x \in V$ we get, by Fubini's Theorem,
\begin{align*}
-u(x)&=\int_{\mathbb{D}} \widehat{R}^{\mathbb{D}
\setminus V}_{G_{\mathbb{D}}(\cdot,y)}(x)d\mu(y)=\int_{\mathbb{D}}\bigl(\int_{\mathbb{D}} G_{\mathbb{D}}(\cdot,y) d\varepsilon_x^{\mathbb{D}\setminus V}\bigr)d\mu(y)\\
&=\int_{\mathbb{D}}\bigl(\int_{\mathbb{D}} G_{\mathbb{D}}(\cdot,y)d\mu(y)\bigr)d\varepsilon_x^{\mathbb{D}\setminus V}=\int_{\mathbb{D}} G_{\mathbb{D}}^{\mu} d\varepsilon_x^{\mathbb{D}\setminus V}=\int_{\mathbb{D}} -u d\varepsilon_x^{\mathbb{D}\setminus V}\\
&=\widehat{R}^{\mathbb{D}\setminus V}_{-u}(x)=-u_{\mathbb{D}\setminus V}^*(x).
\end{align*}

By Proposition \ref{prop3.6}, we now find that $u$ is $\calf$-maximal on $V$.

Combining the $\calf$-local $\calf$-maximalities, we see that $u$
is $\calf$-locally $\calf$-maximal on all of $\mathbb{D}$.

\smallskip

However, $u$ is not maximal on $\mathbb{D}$, since maximal
functions on $\mathbb{D}$ are harmonic, and $u$ is not harmonic as
$\Delta u=-\mu \neq 0$. By Proposition \ref{prop2.3}, this means
that $u$ cannot be $\calf$-maximal on $\mathbb{D}$ either. So we
have found an $\calf$-locally $\calf$-maximal function which is
not $\calf$-maximal.

\end{example}

\begin{remark}
The example above also shows that the converse to \ref{invariant geeft locmax} does not hold. The function $u$ is a finely subharmonic function on the $f$-open set $\mathbb{D} \subset \mathbb{C}$, such that $u \leq 0$ and $u$ is $\calf$-locally $\calf$-maximal. However, $-u$ cannot be invariant. Otherwise, this non-negative finely superharmonic function can be written as $-u=G^{\mu}_{\mathbb{D}}+0=G^0_{\mathbb{D}}+(-u)$. But then the uniqueness of such a decomposition would imply $\mu=0$, which is not the case.
\end{remark}

Let $f\in \PSH(\Omega)$ for a Euclidean open set $\Omega$. We
denote by $\NP(dd^cf)^n$ the Borel measure on $\Omega$ defined by
$$\NP(dd^cf)^n(E)=\lim_{j\to +\infty}\int_{E\cap\{f>-j\}}(dd^cf_j)^n$$
for any Borel subset $E$ of $\Omega$, where $f_j=\max(f,-j)$. The
measure $\NP(dd^c f)^n$ does not charge the  pluripolar subsets of
$\Omega$. We have seen in \cite{EW} that $(dd^c(f|U))^n$
$=NP(dd^cf)^n$ on the ${\cal F}$-open $U=\{f>-\infty\}$.

Let us recall the following result of Bedford and Taylor
\cite[Proposition 4.4]{BT}:

\begin{prop}\label{prop3.25}
For any  compact subset $K$ of $\Omega\setminus\{f=-\infty\}$ we have
$$\NP(dd^cf)^n(K)=\lim_{j\to +\infty}\int_K(dd^cf_j)^n.$$
\end{prop}

The following result seems to be new and gives a positive answer
to a question of \cite[page 13, Problem 6]{BL3} for finite
plurisubharmonic functions:

\begin{theorem}\label{thm3.26}
Let $f\in \PSH(\Omega)$ be finite. If $f$ is locally maximal, then
$f$ is maximal.
\end{theorem}

\begin{proof}
Since $f$ is a locally maximal psh function in $\Omega$, it is
also ${\cal F}$-locally $\calf$-maximal. Hence by Theorem
\ref{cor3.21}, we find that $(dd^c f)^n=0$. We also have
$\NP(dd^cf)^n=(dd^cf)^n$ by \cite[Remark 4.7]{EW}, so $\NP(dd^c
f)^n=0$. Let $\varphi$ be a non-negative continuous function on
$\Omega$ with compact support $K$, then we have
$$0\le \int\varphi(dd^c f_j)^n\le \sup_{x\in K}\varphi(x)\int_K(dd^c f_j)^n.$$
Hence, letting $j\to +\infty$, Proposition \ref{prop3.25} implies that
$$\lim_{j\to +\infty}\int\varphi (dd^c f_j)^n=0.$$
We then deduce from this equality that
$(dd^c f_j)^n$ converges weakly to  $0$. From \cite[Theorem 4.4]{BL1},
it follows that $f$ is maximal.
\end{proof}

\section*{Acknowledgements}

We wish to thank Jan Wiegerinck and the reviewers for all their
insightful comments, suggestions and questions.

\end{document}